\documentclass[review]{elsarticle}

\usepackage{lineno,hyperref}
\modulolinenumbers[5]

\usepackage{amssymb}
\usepackage[english]{babel}
\usepackage{graphicx}
\usepackage{amsmath}
\usepackage{subfig} 

\newtheorem{theorem}{Theorem}[section]

\newtheorem{conjecture}[theorem]{Conjecture}
\newtheorem{corollary}[theorem]{Corollary}

\newtheorem{definition}[theorem]{Definition}

\newtheorem{lemma}[theorem]{Lemma}

\newenvironment{proof}[1][Proof]{\noindent\textbf{#1: }}{\hspace{\stretch{1}}\rule{0.5em}{0.5em}}

\begin{document}

\begin{frontmatter}

\title{A Proof of Erd\"{o}s - Faber - Lov\'{a}sz Conjecture
\footnote{This paper is presented in the 23$^{rd}$ International Conference of Forum for Interdisciplinary Mathematics (FIM) on ``Interdisciplinary Mathematical, Statistical and Computational Techniques - 2014'' organized by the NITK, Surathkal, INDIA, from 18/12/2014 to 20/12/2014.}}

\author{Suresh M. H.}
\ead{smhegde@nitk.ac.in}
\author{ V. V. P. R. V. B. Suresh Dara\corref{mycorrespondingauthor}}
\ead{suresh.dara@gmail.com}
\address{Department of Mathematical and Computational Sciences,\\
National Institute of Technology Karnataka, Surathkal,\\
Mangalore - 575025, India.}
\cortext[mycorrespondingauthor]{Corresponding author}

\begin{abstract}
In 1972, Erd\"{o}s - Faber - Lov\'{a}sz (EFL) conjectured that, if $\textbf{H}$ is a linear hypergraph consisting of $n$ edges of cardinality $n$, then it is possible to color the vertices with $n$ colors so that no two vertices with the same color are in the same edge. In 1978, Deza, Erd\"{o}s and  Frankl had given an equivalent version of the same for graphs: Let $G= \bigcup _{i=1}^{n} A_i$ denote a graph with $n$ complete graphs $A_1, A_2,$ $ \dots , A_n$, each having exactly $n$ vertices and have the property that every pair of complete graphs has at most one common vertex, then the chromatic number of $G$ is $n$. 

The clique degree $d^K(v)$ of a vertex $v$ in $G$ is given by $d^K(v) = |\{A_i: v \in V(A_i), 1 \leq i \leq n\}|$. 
In this paper we give an algorithmic proof of the conjecture using the symmetric latin squares and clique degrees of the vertices of $G$.
\end{abstract}

\begin{keyword}
\texttt Chromatic number \sep Erd\"{o}s - Faber - Lov\'{a}sz conjecture \sep Latin squares
\MSC[2010]  05A05 \sep 05B15  \sep 05C15
\end{keyword}

\end{frontmatter}


\section{Introduction}

One of the famous conjectures in graph theory is Erd\"{o}s - Faber - Lov\'{a}sz conjecture. It states that if $\textbf{H}$ is a linear hypergraph consisting of $n$ edges of cardinality $n$, then it is possible to color the vertices of $\textbf{H}$ with $n$ colors so that no two vertices with the same color are in the same edge \cite{berge1990onvizing}. Erd\"{o}s, in 1975, offered 50 pounds \cite{erdHos1975problems, arroyo2008dense} and in 1981, offered 500USD \cite{erdHos1981combinatorial, jensen2011graph} for the proof or disproof of the conjecture. Kahn \cite{kahn1992coloring} showed that the chromatic number of $\textbf{H}$ is at most $n+o(n)$. Jakson et al.  \cite{jackson2007note} proved that the conjecture is true when the partial hypergraph $S$ of $\textbf{H}$ determined by the edges of size at least three can be $\Delta_S$-edge-colored and satisfies $\Delta_S \leq 3$. In particular, the conjecture holds when $S$ is unimodular and $\Delta_S \leq 3$. Viji Paul et al. \cite{vijipaul2012dmsaa} established the truth of the conjecture for all linear hypergraphs on $n$ vertices with $\Delta(\textbf{H}) \leq \sqrt{n + \sqrt{n} +1}$. Sanhez - Arrayo \cite{arroyo2008dense} proved the conjecture for dense hypergraphs. Faber \cite{faber2010uniformregular} proves that for fixed degree, there can be only finitely many counterexamples to EFL on this class (both regular and uniform) of hypergraphs. We consider the equivalent version of the conjecture for graphs given by Deza, Erd\"{o}s and  Frankl in 1978 \cite{deza1978intersection, arroyo2008dense, jensen2011graph, mitchem2010arscombin}.


\begin{conjecture}\label{EFL}
Let $G= \bigcup _{i=1}^{n} A_i$ denote a graph with $n$ complete graphs $(A_1, A_2,$ $ \dots , A_n)$, each having exactly $n$ vertices and have the property that every pair of complete graphs has at most one common vertex, then the chromatic number of $G$ is $n$.
\end{conjecture}

%


\begin{definition}\label{d1}
Let $G= \bigcup _{i=1}^{n} A_i$ denote a graph with $n$ complete graphs $A_1, A_2,$ $ \dots , A_n$, each having exactly $n$ vertices and the property that every pair of complete graphs has at most one common vertex. The clique degree $d^K(G)$ of a vertex $v$ in $G$ is given by $d^K(v) = |\{A_i: v \in V(A_i), 1 \leq i \leq n\}|$. The maximum clique degree $\Delta ^K(G)$ of the graph $G$ is given by $\Delta ^K(G)=max_{v\in V(G)}d^K(v)$.
\end{definition}


From the above definition one can observe that degree of a vertex in hypergraph is same as the clique degree of a vertex in a graph.

\begin{definition}\label{d3}
Let $G_1$ and $G_2$ be two vertex disjoint graphs, and let $x_1, x_2$
be two vertices of $G_1, G_2$ respectively. Then, the graph $G(x_1x_2)$ obtained by merging the vertices $x_1$ and $x_2$ into a single vertex is called the concatenation of $G_1$ and $G_2$ at the points $x_1$ and $x_2$ (see \cite{kundu1980reconstruction}).
\end{definition}

\begin{definition}\label{d2}
A \textit{latin square} is an $n \times n$ array containing $n$ different symbols such that each symbol appears exactly once in each row and once in each column. Moreover, a latin square of order $n$ is an $n \times n$ matrix $M=[m_{ij}]$ with
entries from an $n$-set $V=\{1, 2, \dots, n\}$, where every row and every column is a permutation of $V$ (see \cite{laywine1998discrete}). If the matrix $M$ is symmetric, then the latin square is called \textit{symmetric latin square}.
\end{definition}

\section{Results}

 We know that a symmetric $n \times n$ matrix is determined by $\frac{n(n + 1)}{2}$ scalars. Using symmetric latin squares we give an $n$-coloring of $H_n$ constructed below. Then using the $n$-coloring of $H_n$ we give an $n$-coloring of all the other graphs $G$ satisfying the hypothesis of Conjecture \ref{EFL}.


\begin{description}
  \item[] \textbf{Construction of $H_n$:}

Let $n$ be a positive integer and $B_1, B_2, \dots, B_n$ be $n$ copies of $K_n$. Let the vertex set $V(B_i)=\{a_{i,1}, a_{i,2}, a_{i,3}, \dots, a_{i,n}\}$, $1 \leq i \leq n$.
  \item[Step 1.] Let $H^1=B_1$.
  \item[Step 2.] Consider the vertices $a_{1,2}$ of $H^1$ and $a_{2,1}$ of $B_2$. Let $b_{1,2}$ be the vertex obtained by the concatenation of the vertices $a_{1,2}$ and $a_{2,1}$. Let the resultant graph be $H^2$.
  \item[Step 3.] Consider the vertices $a_{1,3}$, $a_{2,3}$ of $H^2$ and $a_{3,1}$, $a_{3,2}$ of $B_3$. Let $b_{1,3}$ be the vertex obtained by the concatenation of  vertices $a_{1,3}$, $a_{3,1}$ and let $b_{2,3}$ be the vertex obtained by the concatenation of vertices $a_{2,3}$, $a_{3,2}$. Let the resultant graph be $H^3$.
    \item[]Continuing in the similar way, at the $n$th step we obtain the graph $H^n = H_n$ (for the sake of convenience we take $H^n$ as $H_n$).
\end{description}


By the construction of $H_n$ one can observe the following:

\begin{enumerate}
\item $H_n$ is a connected graph and satisfying the hypothesis of Conjecture \ref{EFL}.
\item $H_n$ has exactly $n$ verticies of clique degree one and $\frac{n(n-1)}{2}$ vertices of clique degree $2$ (each $B_i$ has exactly $(n-1)$ vertices of clique degree $2$ and one vertex of clique degree one, $1 \leq i \leq n$).
\item $H_n=\bigcup_{i=1}^{n}B_i$, where $B_i = A_i$ and $B_i$, $B_j$ have exactly one common vertex for $1\leq i < j \leq n$.
\item $H_n$ has exactly $\frac{n(n+1)}{2}$ vertices.
\item One can observe that in a connected graph $G$ if clique degree increases the number of vertices also increases, from this it follows that, $H_n$ is the graph with minimum number of vertices satisfying the hypothesis of Conjecture \ref{EFL}. If all the vertices of $G$ are of clique degree one, then $G$ will have $n^2$ vertices. Thus, $\frac{n(n+1)}{2} \leq |V(G)| \leq n^2$.
\end{enumerate}



\begin{lemma}\label{ll2}
If $G$ is a graph satisfying the hypothesis of Conjecture \ref{EFL}, then $G$ can be obtained from $H_n$ for some $n$ in $\mathbb{N}$.
\end{lemma}

\begin{proof}
Let $G$ be a graph satisfying the hypothesis of Conjecture \ref{EFL}. Let $b_x$ be the new labeling to the vertices $v$ of clique degree greater than one in $G$, where $x=\{i$\ : vertex $v$ is in $A_i\}$. Define $N_i=\{b_x:|x|=i\}$ for $i=2, 3, \dots, n$. Then the graph $G$ is constructed from $H_n$ as given below:

\textbf{Step 1:} For every common vertex $b_{i,j}$ in $H_n$ which is not in $N_2$, split the vertex $b_{i,j}$ into two vertices $u_{i,j},u_{j,i}$ such that vertex $u_{i,j}$ is adjacent only to the vertices of $B_i$ and the vertex $u_{j,i}$ is adjacent only to the vertices of $B_j$ in $H_n$. 

\textbf{Step 2:} For every vertex $b_x$ in $N_i$ where $i=3, 4, \dots, n$, merge the vertices $u_{l_1,l_2}, u_{l_2,l_3}, \dots, u_{l_{m-1},l_m}, u_{l_m,l_1}$ into a single vertex $u_x$ in $H_n$ where $l_i\in x$ and $l_i < l_j$ for $i<j$. 

Let $G'$ be the graph obtained in Step 2. Let $V(B_i')$, $V(A_i')$ be the set of all clique degree 1 vertices of $B_i$ of $G'$, $A_i$ of $G$ respectively, $1\leq i \leq n$. 
Thus by splitting all the common vertices of $H_n$ which are not in $N_2$ and merging the vertices of $H_n$ corresponding to the vertices in $N_i, i\geq 3$, we get the graph $G'$. One can observe that $|V(A_i')|=|V(B_i')|$, $1 \leq i \leq n$. Define a function $f: V(G) \rightarrow V(G')$ by

\begin{align*}
f(b_{i,j}) 	& 	=  b_{i,j}	 & 	\text{ for } b_{i,j} \in N_2\\
f(b_{i_1,i_2, \dots i_k}) &	=  u_{i_1,i_2, \dots i_k} 	& \text{ for } b_{ii_1,i_2, \dots i_k} \in \cup_{i=3}^n N_i\\
f|_{V(A_i')} 	& 	=  g_i 	 & (\text{any 1-1 map } g_i:V(A_i') \rightarrow V(B_i')), \text{ for }1\leq i \leq n
\end{align*}

One can observe that $f$ is an isomorphism from $G$ to $G'$.
\end{proof}

From Lemma \ref{ll2}, one can observe that in $G$ there are at most ${\frac{n(n-1)}{2}}$ common vertices.


Let $G$ be the graph satisfying the hypothesis of Conjecture \ref{EFL}. Let $\hat{H}$ be the graph obtained by removing the vertices of clique degree one from graph $G$. i.e. $\hat{H}$ is the induced subgraph of $G$ having all the common vertices of $G$.


\begin{lemma}\label{l1}
The chromatic number of $H_n$ is $n$.
\end{lemma}

\begin{proof}
Let $H_n$ be the graph defined as above. Let $M$ (given below) be an $n \times n$ matrix in which an entry $m_{i,j}=b _{i,j}$, is a vertex of $H_n$, belonging to both $B_i, B_j$ for $i \neq j$ and $m_{ii} = a_{i,i}$ is the vertex of $H_n$ which belongs to $B_i$. i.e.,

M=$\left(
  \begin{array}{ccccc}
    a_{1,1} & b_{1,2} & b_{1,3} & \dots & b_{1,n} \\
    b_{1,2} & a_{2,2} & b_{2,3} & \dots & b_{2,n} \\
    b_{1,3} & b_{2,3} & a_{3,3} & \dots & b_{3,n} \\
    \vdots & \vdots & \vdots & \ddots & \vdots \\
    b_{1,n} & b_{2,n} & b_{3,n} & \dots & a_{n,n} \\
  \end{array}
\right)$

\noindent Clearly $M$ is a symmetric matrix. We know that, for every $n$ in $\mathbb{N}$ there is a symmetric latin square (see \cite{ye2011number}) of order $n \times n$. Bryant and Rodger \cite{bryant2004completion} gave a necessary and sufficient condition for the existence of an $(n-1)$-edge coloring of $K_n$ (n even), and $n$-edge coloring of $K_n$ (n odd) using symmetric latin squares. Let $v_1, v_2, \dots, v_n$ be the vertices of $K_n$ and $e_{ij}$ is the edge joining the vertices $v_i$ and $v_j$ of $K_n$, where $i<j$, then arrange the edges of $K_n$ in the matrix form $A = [a_{i,j}]$ where $a_{i,j} = e_{i,j}$, $a_{j,i} = e_{i,j}$ for $ i<j$ and $a_{i,i} =0$ for $1 \leq i \leq n$, we have
$A=\left(
  \begin{array}{ccccc}
    0 & e_{1,2} & e_{1,3} & \dots & e_{1,n} \\
    e_{1,2} & 0 & e_{2,3} & \dots & e_{2,n} \\
    e_{1,3} & e_{2,3} & 0 & \dots & e_{3,n} \\
    \vdots & \vdots & \vdots & \ddots & \vdots \\
    e_{1,n} & e_{2,n} & e_{3,n} & \dots & 0 \\
  \end{array}
\right)$

and let $V$ is a matrix given by

$V=\left(
  \begin{array}{ccccc}
    v_1 & 0 & 0 & \dots & 0 \\
    0 & v_2 & 0 & \dots & 0 \\
    0 & 0 & v_3 & \dots & 0 \\
    \vdots & \vdots & \vdots & \ddots & \vdots \\
    0 & 0 & 0 & \dots & v_n \\
  \end{array}
\right)$. Then, define a matrix $A'$ as

$A' = A+V = \left(
  \begin{array}{ccccc}
    v_1 & e_{1,2} & e_{1,3} & \dots & e_{1,n} \\
    e_{1,2} & v_2 & e_{2,3} & \dots & e_{2,n} \\
    e_{1,3} & e_{2,3} & v_3 & \dots & e_{3,n} \\
    \vdots & \vdots & \vdots & \ddots & \vdots \\
    e_{1,n} & e_{2,n} & e_{3,n} & \dots & v_n \\
  \end{array}
\right)$.

Let $C=[c_{i,j}]$ be a matrix where $c_{i,j}$ ($i\neq j$), is the color of $e_{i,j}$ (i.e., $c_{i,j}=c(e_{i,j})$) and $c_{i,i}$ is the color of $v_i$. We call $C$ the color matrix of $A'$. Then $C$ is the symmetric latin square (see\cite{bryant2004completion}). As the elements of $M$ are the vertices of $H_n$, one can assign the colors to the vertices of $H_n$ from the color matrix $C$, by the color $c_{i,j}$ (where $c_{i,j}$ denotes the value at the $(i, j)$-th entry in the color matrix $C$), for $ i, j =1, 2, \dots,n$ and $i \neq j$ to the vertex $b_{i,j}$ in $H_n$ and the color $c_{i,i}$  (where $c_{i,i}$ denotes the value at the $(i, i)$-th entry in the color matrix $C$), for $i = 1, 2, \dots n$ to the vertex $a_{i,i}$ in $H_n$. Hence $H_n$ is $n$ colorable.
\end{proof}

As $H_n$ is the graph satisfying the hypothesis of Conjecture \ref{EFL}. With using the coloring of $H_n$ which is the graph satisfying the hypothesis of Conjecture \ref{EFL} we extend the $n$-coloring of all possible graphs $G$ satisfying the hypothesis of Conjecture \ref{EFL}.

\begin{theorem}\label{t7}
If $G$ is a graph satisfying the hypothesis of Conjecture \ref{EFL}, then $G$ is $n$-colorable.
\end{theorem}

\begin{proof}
Let $G$ be a graph satisfying the hypothesis of Conjecture \ref{EFL}. Let $\hat{H}$ be the induced subgraph of $G$ consisting of the vertices of clique degree greater than one in $G$. For every vertex $v$ of clique degree greater than one in $G$, label the vertex $v$ by $u_A$ where $A=\{i:v\in A_i; i =1, 2, \dots, n\}$.
Define
$X=\{b_{i,j}:A_i \cap A_j = \emptyset \}$, $X_i=\{v\in G: d^K(v)=i\}$ for $i=1, 2, \dots, m$.

Let $1, 2, \dots, n$ be the $n$-colors and $C$ be the color matrix( of size $n \times n$) as defined in the proof of Lemma \ref{l1}. The following construction applied on the color matrix $C$, gives a modified color matrix $C_M$, using which we assign the colors to the graph $\hat{H}$. Then this coloring can be extended to the graph $G$. Construct a new color matrix $C_1$ by putting $c_{i,j} = 0, c_{j,i} = 0$ for every $b_{i,j}$ in $X$. Also, let $c_{i,i} = 0$ for each $i=1,2,\dots,n$

\textbf{Construction:}

Let $T=\cup_{i=3}^{n}X_i$, $P= \emptyset$, $T'' = X_2$ and $P'' = \emptyset$.

\begin{description}

    \item[Step 1:]
If $T = \emptyset$, let $C_m$ be the color matrix obtained in Step 4 and go to Step 5. Otherwise, choose a vertex $u_{i_1,i_2, \dots, i_m}$ from $T$, where $i_1 <i_2 < \dots < i_m$, and then choose $\binom{m}{2}$ vertices $b_{i_1,i_2}$, $b_{i_1,i_3}$, $\dots$, $b_{i_1,i_m}$, $b_{i_2,i_3}$, $\dots$, $b_{i_{m-1},i_m}$ from $V(H_n)$ corresponding to the set $\{i_1, i_2,  \dots, i_m\}$. Take $T' = \{b_{i_1,i_2}, b_{i_1,i_3},\dots, b_{i_1,i_m}, b_{i_2,i_3}, \dots, b_{i_{m-1},i_{m}}\}$ and $P' = \emptyset$. Let $T_1' = \{b_{i,j}: b_{i,j} \in T',  c(b_{i,j})$ appears more than once in the $ i^{th}$ row or $ j^{th}$ column in $ C \}$ and $T_2' = \{ b_{i,j} : b_{i,j} \in T', c(b_{i,j}) $ appears exactly once in the $ i^{th}$ row and $ j^{th}$ column in $ C \}$. If $T_1' \neq \emptyset$ choose a vertex $b_{s,t}$ from $T_1'$, otherwise choose a vertex $b_{s,t}$ from $T_2'$. Then add the vertex $b_{s,t}$ to $P'$ and remove it from $T'$. Go to Step 2.

\item[Step 2:] If $T_2' \neq \emptyset$ go to Step 3. Otherwise, choose a vertex $b_{i_{m-1},i_m}$ from $T_1'$. Let $A=\{c_{i,j}: c_{i,j} \neq 0; i = i_{m-1}, 1\leq j \leq n\}$, $B=\{c_{i,j} : c_{i,j} \neq 0 ; j =i_m, 1 \leq i \leq n\}$. 
If $|A\cap B| <n$ then, 
construct a new color matrix $C_2$, replacing $c_{i_{m-1},i_m}$, $c_{i_m,i_{m-1}}$ by $x$, where $x \in \{1, 2, \dots, n\} \setminus A\cup B$. Then add the vertex $b_{i_{m-1},i_m}$ to $T_2'$ and remove it from $T_1'$. Go to Step 3.
Otherwise choose a color $x$ which appears exactly once either in $i_{m-1}^{th}$ row or in $i_m^{th}$ column of the color matrix and construct a new color matrix $C_2$ replacing $c_{i_{m-1},i_m}$, $c_{i_m,i_{m-1}}$ by $x$. Then add the vertex $b_{i_{m-1},i_m}$ to $T_2'$ and remove it from $T_1'$. Go to Step 3.

 \item[Step 3:] If $T' = \emptyset$, then add the vertex $u_{i_1,i_2, \dots, i_m}$ to $P$ and remove it from $T$, go to Step 1.  Otherwise, if $T' \cap T_1' \neq \emptyset$ choose a vertex  $b_{i,j}$ from $T' \cap T_1'$, if not choose a vertex  $b_{i,j}$ from $T' \cap T_2'$. Go to Step 4.

 \item[Step 4:] Let $c(b_{i,j})=x,\  c(b_{s,t})=y$. If $c(b_{i,j})=c(b_{s,t})$, then add the vertex $b_{i,j}$ to $P'$ and remove it from $T'$. Go to Step 3. Otherwise, let $A=\{ c_{l,m} : c_{l,m} = x\}$, $B=\{c_{l,m} : c_{l,m}=y\} \setminus \{c_{l,m}, c_{m,l}: b_{l,m} \in P', l<m\}$. Construct a new color matrix $C_3$ by putting $c_{l,m} = y$ for every $c_{l,m}$ in $A$ and $c_{l,m}=x$ for every $c_{l,m}$ in B. Then add the vertex $b_{i,j}$ to $P'$ and remove it from $T'$. Go to Step 3.

  \item[Step 5:] If $T'' = \emptyset$ consider $C_M = C_{m_1}$ stop the process. Otherwise, choose a vertex $u_{i,j}$ from $T''$ and go to Step 6.
  
  \item[Step 6:] If $c_{i,j}$ appears exactly once in both $i^{th}$ row and $j^{th}$ column of the color matrix $C_m$, then add the vertex $b_{i,j}$ to $P''$ and remove it from $T''$, go to Step 5. Otherwise let $A=\{c_{i,j}: c_{i,j} \neq 0; 1\leq j \leq n\}$, $B=\{c_{i,j} : c_{i,j} \neq 0 ; 1 \leq i \leq n\}$. Construct a new color matrix $C_{m_1}$ by putting $x$ in $c_{i,j}$, $c_{j,i}$ where $x \in \{1, 2, \dots, n\} \setminus A\cup B$. Then add the vertex $u_{i,j}$ to $P''$ and remove it from $T''$, go to Step 5.

\end{description}

Thus, in step 6, we get the modified color matrix $C_M$.
Then, color the vertex $v$ of $\hat{H}$ by $c_{i,j}$ of $C_M$, whenever $v \in A_i \cap A_j$. Then, extend the coloring of $\hat{H}$ to $G$ by assigning the remaining colors which are not used for $A_i$ from the set of $n$-colors, to the vertices of clique degree one in $A_i$, $ 1 \leq i \leq n$. Thus $G$ is $n$-colorable.
\end{proof}

\begin{corollary}\cite{arroyo2008dense}
Consider a linear hypergraph $\textbf{H}$ consisting of $n$ edges each of size at most $n$ and $\delta(\textbf{H}) \geq 2$. If $H$ is dense then $\chi(\textbf{H}) \leq n$.
\end{corollary}

\section*{References}

\bibliography{mydatabase}

\end{document}